\documentclass[reqno,12pt]{amsart}

\usepackage{epsf}
\usepackage{graphics}
\usepackage{graphicx}
\usepackage{amssymb}
\usepackage{amsmath}

\date{}

\theoremstyle{plain}
\newtheorem{theorem}{Theorem}

\newtheorem{proposition}{Proposition}

\newtheorem{rem}{Remark}

\theoremstyle{definition}

\theoremstyle{remark}

\newtheorem*{remark}{Remark}

\def\C{{\mathbb C}}
\def\N{{\mathbb N}}
\def\Z{{\mathbb Z}}

\def\F{{\mathbb F}}

\title{Two strand twisting}

\author{L.~A'Campo, S.~Baader, L.~Ferretti, L.~Ryffel}

\begin{document}

\begin{abstract}
We prove that fibred knots cannot be untied with $\bar{t}_{2k}$-moves, for all $k \geq 2$. More generally, we give an upper bound on the number of two strand twist operations that allow to untie a knot with non-trivial HOMFLY polynomial, in terms of the minimal crossing number, and the braid index. As a by-product, we prove that the braid index of a two-bridge knot cannot be lowered by applying $t_{2k}$-moves, for all but finitely many $k \in \N$.
\end{abstract}


\maketitle

\section{Introduction}

Twisting is a family of local operations on oriented links in $S^3$, introduced by Ralph Fox in the late fifties~\cite{Fo}. The special case of two strand twisting comes in two families called $t_m$-moves and $\bar{t}_m$-moves. The effect of a $t_m$-move (resp. $\bar{t}_m$-move) is best described on oriented link diagrams, where it inserts $m$ consecutive positive crossings into two parallel strands with equal orientations (resp. opposite orientations), as in the two strand braid $\sigma_1^m \in B_2$. We will restrict our attention to even numbers $m \in \N$, since we want all moves to preserve the number of link components. 
The two simplest moves, $t_2$ and $\bar{t}_{2}$, are better known as (positive) crossing changes. While every knot can be untied by a finite sequence of crossing changes, there exist obstructions for knots to be unknotted by higher order $t_{2k}$- and $\bar{t}_{2k}$-moves. In particular, the Alexander-Conway polynomial of knots with coefficients reduced modulo~$k$, $\nabla_K(z) \in \Z/k\Z[z]$, is invariant under $\bar{t}_{2k}$-moves~\cite{Fo, Pr}. This has a remarkable consequence for fibred knots, since these have monic Alexander-Conway polynomial~\cite{Neu}. 

\begin{theorem}
\label{fibredknots}
Let $K$ be a non-trivial fibred knot. For all $k \geq 2$, the knot~$K$ is not related to the trivial knot by a finite sequence of
$(\bar{t}_{2k})^{\pm 1}$-moves.
\end{theorem}

As far as the authors know, this statement never found its way into the literature, most likely since Neuwirth's theory of fibred knots was developed after Fox' note on congruence classes of knots.

In contrast with Theorem~\ref{fibredknots}, the situation is very different with $t_{2k}$-moves. For all $n \in \N$, there exists a fibred knot that can be unknotted by $t_{2k}$-moves, for all $k \leq n$ , for example the closure of the braid
$$\sigma_1^{1+\text{lcm}(2,4,6 \ldots 2n)} \in B_2,$$
as shown in Figure~1 for $n=4$.

\smallskip
\begin{figure}[htb]
\begin{center}
\raisebox{-0mm}{\includegraphics[scale=0.7]{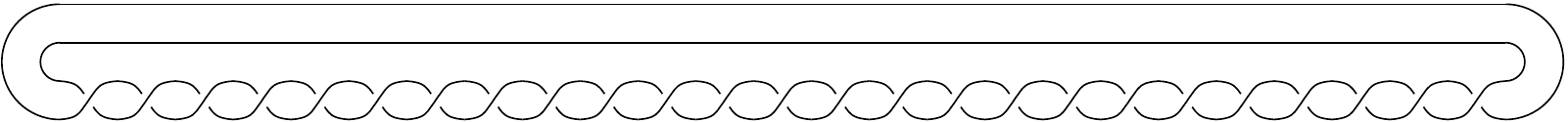}}
\caption{Closure of the braid $\sigma_1^{25}$, which can be unknotted by $t_{2k}$-moves, for $k=1,2,3,4,6,12,13$.}
\end{center}
\end{figure}

Nevertheless, the two families of two strand twisting operations share a common property: most knots can be unknotted by finitely many different types of $t_{2k}$-moves and $\bar{t}_{2k}$-moves only. Our main result makes this quantitative, in terms of well-known link invariants. Let $c(K) \in \N$ and $b(K) \in \N$ be the minimal crossing number and the braid index of a link~$K$, respectively. The latter is defined as the minimal number of strands among all braids whose closure represents the link~$K$. Furthermore, let $P_K(a,z) \in \Z[a^{\pm 1},z^{\pm 1}]$ be the HOMFLY polynomial of~$K$, defined in the next section.

\begin{theorem}
\label{finiteness}
Let $K$ be a knot with $P_K(a,z) \neq 1$.
\begin{enumerate}
\item The set $\{k \geq 3 \mid \text{$K$ can be unknotted by $(t_{2k})^{\pm 1}$-moves}\}$ has at most $c(K)-1$ elements.
\item The set $\{k \geq 2 \mid \text{$K$ can be unknotted by $(\bar{t}_{2k})^{\pm 1}$-moves}\}$ has at most $b(K)-1$ elements.
\end{enumerate}
\end{theorem}

The condition $P_K(a,z) \neq 1$ seems quite generic, and is possibly even satisfied by all non-trivial knots. Nevertheless, it would be great to derive bounds of the above kind for all non-trivial knots, for example by using Khovanov homology, which is known to detect the trivial knot~\cite{KM}. The mere existence of finite upper bounds in Theorem~\ref{finiteness} might also admit a geometric proof, due to its resemblance with Thurston's hyperbolic Dehn surgery theorem~\cite{Th}. However, the latter deals with fixed twist regions, which provides an a priori weaker statement.

The technique used in our proof allows a precise determination of the set of unknotting moves for certain classes of knots, such as two strand torus knots and twist knots, see Proposition~\ref{twist}. More importantly, we obtain the following refined result for two-bridge knots. 

\begin{proposition}
\label{twobridge}
Let $K$ be a two-bridge knot. For all but finitely many $k \in \N$, all knots $K'$ that are related to $K$ by a finite sequence of $(t_{2k})^{\pm 1}$-moves satisfy
$$b(K') \geq b(K).$$
\end{proposition}

Results providing a lower bound for the braid index within an equivalence class of knots are not so common; an interesting one was recently derived by Feller and Hubbard: closures of quasipositive braids with sufficiently many full twists are not concordant to quasipositive knots with a strictly smaller braid index~\cite{FH}.

The proofs of Theorem~\ref{finiteness} and Proposition~\ref{twobridge} are presented in Sections~3 and~4; the next section contains the necessary fundamentals about the HOMFLY polynomial. We would like to emphasise that most of the statements included here are applications of Przytycki's theory on two strand twisting~\cite{Pr}, but none of them seem to appear in the literature so far.

\section{HOMFLY polynomial and twisting}

The HOMFLY polynomial $P_K(a,z) \in \Z[a^{\pm 1},z^{\pm 1}]$ of links~$K$ is defined by the following skein relation, together with the normalisation $P_O(a,z)=1$ for the trivial knot~$O$~\cite{HOMFLY}:
$$a^{-1} P_{\widehat{\beta \sigma_i^2}}(a,z)-aP_{\widehat{\beta}}(a,z)=zP_{\widehat{\beta \sigma_i}}(a,z).$$
Here $\widehat{\beta}$ stands for the closure of a braid with $n$ strands, $\beta \in B_n$, and $\sigma_i \in B_n$ denotes any standard generator in the braid group $B_n$.
As observed by various authors, the skein relation is well-suited to compute the effect of $t_m$-moves on links, see for example~\cite{NS,Pr}. In fact, the following Proposition is basically a reformulation of Corollaries~1.2 and~1.8 in~\cite{Pr}, except for the case $k=2$.

\begin{proposition}
\label{invariance}
Let $K$ be a knot, and let $\zeta_{2k} \in \C$ be a primitive $2k$-th root of unity.
\begin{enumerate}
\item[(i)] If $K'$ is a knot obtained from $K$ by a positive $t_{2k}$-move with $k \geq 3$, then
$$P_{K'}(a,\zeta_{2k}-\zeta_{2k}^{-1})=a^{2k}P_{K}(a,\zeta_{2k}-\zeta_{2k}^{-1}).$$
\item[(ii)] If $K'$ is a knot obtained from $K$ by a positive $t_{4}$-move, then
$$P_{K'}(a,0)=a^4P_{K}(a,0) \in \F_2[a^{\pm 1}].$$
\item[(iii)] If $K'$ is a knot obtained from $K$ by a positive $\bar{t}_{2k}$-move with $k \geq 2$, then
$$P_{K'}(\zeta_{2k},z)=P_{K}(\zeta_{2k},z).$$
\end{enumerate}
\end{proposition}

\begin{proof}
For the first two statements, let $L_0,L_1,L_2,\ldots$ be a family of oriented link diagrams, which coincide except in a single twist region consisting of two parallel oriented strands with a certain number of positive crossings -- $n$ for the link diagram $L_n$ -- again as in the two strand braid $\sigma_1^n \in B_2$. The skein relation for the HOMFLY polynomial translates into the recursion
$$P_{L_{n+1}}(a,z)=a^2P_{L_{n-1}}(a,z)+azP_{L_{n}}(a,z),$$
which admits the following matrix representation:
$$
\begin{pmatrix}
P_{L_{n}}  \\
P_{L_{n+1}}
\end{pmatrix}
=
\begin{pmatrix}
0 & 1 \\
a^2 & az 
\end{pmatrix}
\begin{pmatrix}
P_{L_{n-1}}  \\
P_{L_{n}}
\end{pmatrix}.
$$
Let $\zeta_{2k} \in \C$ be a primitive $2k$-th root of unity with $k \geq 3$. The specialisation $z=\zeta_{2k}-\zeta_{2k}^{-1} \neq 0$ gives rise to a recursion matrix
$$ 
M=
\begin{pmatrix}
0 & 1 \\
a^2 & a(\zeta_{2k}-\zeta_{2k}^{-1})
\end{pmatrix}
$$
with $\text{tr}(M)=a(\zeta_{2k}-\zeta_{2k}^{-1})$ and $\det(M)=-a^2$, hence
$$ 
M^{2k}
=
\begin{pmatrix}
a^{2k} & 0 \\
0 & a^{2k} 
\end{pmatrix}.
$$
The last step requires $M$ to have two distinct eigenvalues $a\zeta_{2k}$, $-a\zeta_{2k}^{-1}$, which is the case for $k \geq 3$, but not for $k=2$.
This implies the first statement of the proposition.

In the case $k=2$, i.e. for $z=2i$, the recursion matrix $M$ has a double eigenvalue $ia$ and is not diagonalisable. However, the equation
$$
M^4
=
\begin{pmatrix}
a^4 & 0 \\
0 & a^4 
\end{pmatrix}
$$
still holds modulo~$4$, which implies that for knots, the reduction of the Laurent polynomial $P_K(a,0) \in \F_2[a^{\pm 1}]$ is still invariant under $t_4$-moves, up to multiplication with powers of $a^4$. Indeed, the above reduction yields
$$P_{L_{5}}(a,2i)=-4ia^5P_{L_{0}}(a,2i)+5a^4P_{L_{1}}(a,2i).$$
Now suppose that $L_1$ is a knot. Then $L_0$ is a two-component link, thus $P_{L_0}(a,z)$ contains a simple pole in~$z$. In turn, the expression $P_{L_0}(a,2i)$ contains a term of the form $1/2i$, which, after multiplication with $-4ia^5$, becomes zero modulo two. This leaves us with the equation $P_{L_{5}}(a,0)=a^4P_{L_{1}}(a,0) \in \F_2[a^{\pm 1}]$, i.e. with the second statement of the proposition.

As for the third statement, we refer the reader to Przytycki's proof of Theorem~1.7 in~\cite{Pr}, keeping the following two facts in mind. First, the sign convention for the skein relation of the HOMFLY polynomial is different in~\cite{Pr}, leading to an overall sign $(-1)^k$ in the formula there. Second, the case $k=2$ is excluded in Corollary~1.8 of the same reference, since the proof contains a fraction with denominator $a+a^{-1}$. However, the only fact needed there is the equality
$$a+a^3+a^5+\ldots+a^{2k-1}=0,$$
which also holds for $k=2$, i.e. for $a=\zeta_4=i$, yielding the desired formula for all $k \geq 2$.
\end{proof}

\begin{remark}
Reductions of the form $P_K(a,N) \in \F_p[a^{\pm 1}]$ have been studied in~\cite{Ba}.
Proposition~\ref{invariance} provides infinitely many reductions of the HOMFLY polynomial, invariant under $t_{2k}$-moves, up to multiplication with powers of $a^{2k}$. Indeed, let $k \geq 3$, and let $p \in 2k\N+1$ be a prime number. Then the finite field $\F_p$ has a primitive $2k$-th root of unity $\zeta_{2k}$, since the multiplicative group of 
$\F_p$ is cyclic of order $p-1$. Let $N=\zeta_{2k}-\zeta_{2k}^{-1} \in \F_p$. The statement of Proposition~\ref{invariance} carries over to the reduction $P_K(a,N) \in \F_p[a^{\pm 1}]$: let $K'$ be a link obtained from a link $K$ by a positive $t_{2k}$-move with $k \geq 3$. For every prime number $p \in 2k\N+1$, and $N=\zeta_{2k}-\zeta_{2k}^{-1} \in \F_p$, the equality
$$P_{K'}(a,N)=a^{2k}P_{K}(a,N)$$
holds in $\F_p[a^{\pm 1}]$.
Thanks to Dirichlet's theorem on arithmetic progressions~\cite{Di}, the number of primes $p$ in $2k\N+1$ is infinite, for each fixed $k \geq 2$. As a consequence, we obtain infinitely many invariant reductions $P_K(a,N) \in \F_p[a^{\pm 1}]$ under $t_{2k}$-moves, provided $k \geq 3$. Similarly, there exist infinitely many reductions of the form $P_K(M,z) \in \F_p[z^{\pm 1}]$ invariant under $\bar{t}_{2k}$-moves, for all $k \geq 2$.
\end{remark}

\section{Bounding the order of untwisting}

The skein relation of the HOMFLY polynomial implies that the specialisation $P_K(a,a^{-1}-a)$ is constantly one. As a consequence, the polynomial $P_K(a,z)$ cannot be of the form $a^m f(z)$, unless $P_K(a,z)=1$.

\begin{proof}[Proof of Theorem~\ref{finiteness}]
For the first statement, let $K$ be a knot with $P_K(a,z) \neq 1$ and let $d=\text{deg}_z(P_K)$. Write
the terms of lowest and highest $a$-degree in $P_K(a,z)$ as $a^m f(z)$ and $a^n g(z)$, respectively, with $m,n \in \Z$, $m<n$, and $f(z),g(z) \in \Z[z]$. Suppose that $K$ is related to the trivial knot by a finite sequence of $(t_{2k})^{\pm 1}$-moves, for some $k \geq 3$. Then, by Proposition~\ref{invariance}, either $f(\zeta_{2k}-\zeta_{2k}^{-1})$ or $g(\zeta_{2k}-\zeta_{2k}^{-1})$ must be zero, since the trivial knot~$O$ satisfies $P_O(a,z)=1$. Therefore, the product $f(z)g(z)$ vanishes at $z=\zeta_{2k}-\zeta_{2k}^{-1}$. The degree bound $\text{deg}(f(z)g(z)) \leq 2d$ implies that the set
$$\{k \geq 3 \mid \text{$K$ can be unknotted by $(t_{2k})^{\pm 1}$-moves}\}$$
has at most $d$ elements, since the minimal polynomial of the purely imaginary number $\zeta_{2k}-\zeta_{2k}^{-1}$ has degree at least two. This yields the first statement, thanks to Franks-Williams and Morton's upper bound for the $z$-degree of the HOMFLY polynomial~\cite{FW,Mo}:
$$\text{deg}_z(P_K) \leq c(K)-1.$$

For the second statement, write
$$P_K(a,z)=h_0(a)+h_1(a)z^2+h_2(a)z^4+\ldots+h_l(a)z^{2l}$$
with $l \geq 1$ and $h_l(a) \neq 0$, since $P_K(a,z) \neq 1$. Suppose that $K$ is related to the trivial knot by a finite sequence of $(\bar{t}_{2k})^{\pm 1}$-moves, for some $k \geq 2$. Then, by Proposition~\ref{invariance}, $h_l(\zeta_{2k})=0$. Therefore, the set
$$\{k \geq 2 \mid \text{$K$ can be unknotted by $(\bar{t}_{2k})^{\pm 1}$-moves}\}$$
has at most as many elements as half the number of roots of the Laurent polynomial $h_l(a)$, again since the minimal polynomial of the number~$\zeta_{2k}$ has degree at least two. As a consequence, the above set has at most $\frac{1}{2}\text{$a$-span}(P_{K}(a,z))$ elements, where $\text{$a$-span}(P_{K}(a,z))$ is the difference of the highest and lowest $a$-degree in $P_K(a,z)$. This yields the second statement, thanks to another inequality by Franks-Williams and Morton~\cite{FW,Mo}:
$$2b(K) \geq \text{$a$-span}(P_K(a,z))+2.$$
\end{proof}

\section{Parallel twisting and braid index}

The inequality of Franks-Williams and Morton~\cite{FW,Mo}, used at the end of the last section, remains true under the specialisation $z=\zeta_{2k}-\zeta_{2k}^{-1}$.
We observe that for all but finitely many $k \in \N$,
$$\text{$a$-span}(P_K(a,\zeta_{2k}-\zeta_{2k}^{-1}))=\text{$a$-span}(P_K(a,z)).$$
Furthermore, the $a$-span of $P_K(a,\zeta_{2k}-\zeta_{2k}^{-1})$ is certainly invariant under multiplication with powers of $a^{2k}$. Therefore, the first item of Proposition~\ref{invariance} implies the following statement.

\begin{proposition}
\label{braidindex}
Let $K$ be a knot and $k \geq 3$. Then every knot $K'$ that is related to $K$ by a finite sequence of $(t_{2k})^{\pm 1}$-moves satisfies
$$2b(K') \geq \text{$a$-span}(P_{K}(a,\zeta_{2k}-\zeta_{2k}^{-1}))+2.$$
\end{proposition}

In the case of two-bridge knots~$K$, Murasugi showed that the above inequality is sharp: $2b(K)=\text{$a$-span}(P_K(a,z))+2$.
This implies Proposition~\ref{twobridge}.

For two special families of two-bridge knots, we obtain even better results: let $K_n$ be the family of twist knots with 2 negative crossings, and $2n$ positive crossings. In Rolfsen's notation~\cite{Ro}, the first four knots of this sequence are $4_1,6_1,8_1,10_1$, see Figure~2 for $n=2$. For convenience, we add the trivial knot $K_0=O$. Furthermore, let $T(2,2n+1)$ be the two strand torus knot with $2n+1$ positive crossings.

\begin{proposition}
\label{twist}
\quad
\begin{enumerate}
\item[(i)] For all $n \geq 1$, the set
$$\{k \geq 1 \mid \text{$K_n$ can be unknotted by $(\bar{t}_{2k})^{\pm 1}$-moves}\}$$
coincides with the set of divisors of $n$. Moreover, every knot $K'$ that is related to $K_n$ by a finite sequence of $(t_{2k})^{\pm}$-moves with $k \geq 2$ satisfies
$$b(K') \geq b(K_n).$$
In particular, the knot $K_n$ cannot be unknotted by $(t_{2k})^{\pm 1}$-moves, except for $k=1$.
\item[(ii)] For all $n \geq 1$, the set
$$\{k \geq 1 \mid \text{$T(2,2n+1)$ can be unknotted by $(t_{2k})^{\pm 1}$-moves}\}$$
coincides with the set of divisors of $n$ and $n+1$. Moreover, the knot $T(2,2n+1)$ cannot be unknotted by $(\bar{t}_{2k})^{\pm 1}$-moves, except for $k=1$.
\end{enumerate}
\end{proposition}

\begin{figure}[htb]
\begin{center}
\raisebox{-0mm}{\includegraphics[scale=1.0]{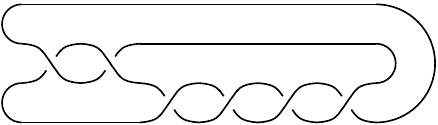}}
\caption{Twist knot $K_2=6_1$.}
\end{center}
\end{figure}

\begin{proof}
For the first statement, suppose that $k$ is a divisor of $n$. Then an obvious sequence of $(\bar{t}_{2k})^{-1}$-moves transforms
$K_n$ into the trivial knot. Next, suppose that $K_n$ can be unknotted by a finite sequence of $(\bar{t}_{2k})^{\pm 1}$-moves, for $k \geq 2$. By Fox' congruence statement mentioned in the introduction, the Alexander-Conway polynomial of $K_n$ must be equal to~$1$ modulo~$k$. An easy computation reveals $\nabla_{K_n}=1-nz^2$, thus~$k$ has to be a divisor of~$n$. For the last part of the first item, we compute the HOMFLY polynomial of $K_n$, via the skein relation:
$$P_{K_{n+1}}(a,z)=a^2P_{K_{n}}(a,z)+azP_{H^{-}}(a,z),$$
for all $n \in \N$, where $H^{-}$ denotes the Hopf link with two negative crossings.
Using $P_{K_0}(a,z)=P_O(a,z)=1$ and $azP_{H^{-}}(a,z)=a^{-2}-1-z^2$, we find that the terms of lowest and highest $a$-degree of $P_{K_n}(a,z)$ are $a^{-2}$ and $a^{2n}$, respectively. As a consequence, for all $k \geq 3$,
$$\text{$a$-span}(P_{K_n}(a,\zeta_{2k}-\zeta_{2k}^{-1}))=\text{$a$-span}(P_{K_n}(a,z)),$$
and the reduction $P_{K_n}(a,0) \in \F_2[a^{\pm 1}]$ also shares the same $a$-span. Now Proposition~\ref{invariance} and Proposition~\ref{braidindex} imply that for all $k \geq 2$, all knots~$K'$ related to $K_n$ by a finite sequence of $t_{2k}$-moves satisfy
$$2b(K') \geq \text{$a$-span}(P_{K'}(a,z))+2 \geq \text{$a$-span}(P_{K_n}(a,z))+2=2b(K_n).$$
The last equality is again a consequence of Murasugi's result on two-bridge knots.

For the second statement, suppose that $k$ is a divisor of $n$ or $n+1$. Then an obvious sequence of $t_{2k}^{-1}$-moves transforms
$T(2,2n+1)$ into the trivial knot, in the guise of $T(2,1)$ or $T(2,-1)$, respectively. This always works for $k=1,2$. Next, suppose that $T(2,2n+1)$ can be unknotted by a finite sequence of $(t_{2k})^{\pm 1}$-moves, for $k \geq 3$. Then, by Proposition~\ref{invariance}, the polynomial $P_{T(2,2n+1)}(a,\zeta_{2k}-\zeta_{2k}^{-1})$ must be a power of $a^{2k}$. An easy induction shows that $P_{T(2,2n+1)}(a,z)$ takes the form $a^{2n}f(z)+a^{2n+2}g(z)$. We conclude that $k$ divides $n$ or $n+1$. The very last statement is an immediate consequence of Theorem~\ref{fibredknots}.
\end{proof}

\bigskip
\noindent
University of Oxford, Radcliffe Observatory, Andrew Wiles Building, Woodstock Rd, Oxford OX2 6GG, UK

\smallskip
\noindent
\texttt{lambert.acampo@maths.ox.ac.uk}

\bigskip
\noindent
Mathematisches Institut, Universit\"at Bern, Sidlerstrasse 5, 3012 Bern, Switzerland

\smallskip
\noindent
\texttt{sebastian.baader@unibe.ch}

\smallskip
\noindent
\texttt{livio.ferretti@unibe.ch}

\smallskip
\noindent
\texttt{levi.ryffel@unibe.ch}

\end{document}